\documentclass[reqno]{amsart}
\usepackage{amssymb}
\usepackage{ifpdf}
\ifpdf
 \usepackage[hyperindex,pagebackref]{hyperref}
\else
 \expandafter\ifx\csname dvipdfm\endcsname\relax
 \usepackage[hypertex,hyperindex,pagebackref]{hyperref}
 \else
 \usepackage[dvipdfm,hyperindex,pagebackref]{hyperref}
 \fi
\fi
\allowdisplaybreaks[4]
\numberwithin{equation}{section}
\theoremstyle{plain}
\newtheorem{theorem}{Theorem}[section]

\newtheorem{lemma}{Lemma}[section]

\theoremstyle{remark}
\newtheorem{remark}{Remark}[section]

\begin{document}

\title[Some generalizations of Feng Qi type inequalities]
{Some generalizations of Feng Qi type integral inequalities on time scales}

\author[L. Yin]{Li Yin}
\address[L. Yin]{Department of Mathematics, Binzhou University, Binzhou City, Shandong Province, 256603, China}
\email{\href{mailto: L. Yin<yinli_79@163.com>}{yinli\_79@163.com}}

\author[V. Krasniqi]{Valmir Krasniqi}
\address[V. Krasniqi]{Department of Mathematics,University of Prishtin\"{e}, Prishtin\"{e}, 10000, Republic of Kosova}
\email{\href{mailto: V. Krasniqi<vail.99@hotmail.com>}{vail.99@hotmail.com}}



\thanks {The first author was supported by NSF of Shandong
Province under grant numbers ZR2012AQ028, and by the
Science Foundation of Binzhou University under grant BZXYL1303}.

\begin{abstract}
In this paper, we provide some new generalizations of Feng Qi type integral inequalities on
time scales by using elementary analytic methods.
\end{abstract}

\maketitle

\section{ Introduction}
The following problem has been posed by Qi in \cite{Q}: under what
 conditions does the inequality \begin{equation}
\int_a^b {f^p(x) dx \geqslant } \left(\int_a^b {f(x)dx} \right)^{p -
1}\end{equation} holds for $p>1?$ Later, this problem has been of
great interest for many mathematicians. M. Akkouchi proved the
following results in \cite[p.~124, Theorem~C]{A}.
 \begin{theorem} Let $[a,b]$ be a closed interval of $\mathbb{R}$
 and $p>1.$ For any continuous function $f(x)$ on $[a,b]$ such that
 $f(a)\geqslant0, f'(x)\geqslant{p},$ we have that \begin{equation}
\int_a^b {f^{p + 2}(x) dx \geqslant } \frac{1} {{(b - a)^{p - 1}
}}\left(\int_a^b {f(x)dx} \right)^{p +
1}.\end{equation}\end{theorem} Then, it has been obtained the
$q$-analogue of the previous result in \cite[Proposition 3.5]{BBS}
as follows.
\begin{theorem}Let $p>1$ be a real number and $f(x)$ be a function defined on
$[a,b]_q$, such that $f(a)\geqslant0, D_q{f(x)}\geqslant{p}$ for all
$x\in{(a,b]_q}.$ Then \begin{equation} \int_a^b {f^{p + 2}(x) d_q x
\geqslant } \frac{1} {{(b - a)^{p - 1} }}\left(\int_a^b {f(qx)d_q x}
\right)^{p + 1}.\end{equation}\end{theorem}
Later, V. Krasniqi and A. S. Shabani obtained some more
sufficient conditions to Qi type $h$-integral inequalities in
\cite{KS}. M. R. S. Rahmat got some $(q,h)$-analogues of integral inequalities on discrete time scales in \cite{R}. L. Yin, Q. M. Luo and F. Qi obtained some Qi type inequalities on time scales in \cite{YLQ}. For more results, we refer the reader to the papers
(\cite{BB}-\cite{BL}, \cite{CK}-\cite{HLP}, \cite{LNH}-\cite{MQ},
\cite{STY}-\cite{Y}). Recently, V. Karasniqi obtained some generalizations of Qi type inequalities in \cite{K}. His main results are following two theorems.

\begin{theorem} If $f$ is a non-negative increasing function on $[a,b]$ and satisfies $f'(x)\geqslant (t-2)(x-a)^{t-3}$ for $t\geqslant3$, then \begin{equation}
\int_a^b {f^t (x)dx - } \left( {\int_a^b {f(x)dx} } \right)^{t - 1}  \geqslant f^{t - 1} (a)\int_a^b {f(x)dx} .
\end{equation}\end{theorem}

\begin{theorem} Let $p\geqslant1$. If $f$ is a non-negative increasing function on $[a,b]$ and satisfies $f'(x)\geqslant p\left( {\frac{{x - a}}{{b - a}}} \right)^{p - 1}
$, then \begin{equation}
\int_a^b {f^{p + 2}(x) d x
- } \frac{1} {{(b - a)^{p - 1} }}\left(\int_a^b {f(x)d x}
\right)^{p + 1}\geqslant f^{p+ 1} (a)\int_a^b {f(x)dx} .
\end{equation}\end{theorem}

The main aim of this paper is to generalize
the above results on time scales.
\section{Notations and lemmas}

\subsection{Notations}

A time scale $\mathbb{T}$ is an arbitrary nonempty closed subset of
the real numbers $\mathbb{R}$. The forward and backward jump
operators $ \sigma,\rho :{\mathbb{T}} \to {\mathbb{T}} $ are defined
by $$ \sigma(t) = \inf \left\{ {s \in {\mathbb{T}}:s
> t} \right\},
$$
$$
\rho (t) = \sup \left\{ {s \in {\mathbb{T}}:s < t} \right\},
$$
where the supremum of the empty set is defined to be the infimum of
$ {\mathbb{T}} $. A point $
 t \in {\mathbb{T}} $ is said to be right-scattered if $
 \sigma(t) > t  $ and right-dense if $
 \sigma(t) = t $, and $
 t \in {\mathbb{T}} $ with $
 t > \inf {\mathbb{T}} $ is said to be left-scattered if $
 \rho (t) < t $ and left-dense if $
 \rho (t) = t. $ 
 A function $
 g:{\mathbb{T}} \to \mathbb{R} $ is said to be rd(ld)-continuous provided $
 g $ is continuous at right(left)-dense points and has finite left(right)-sided limits at left(right)-dense points in $
 {\mathbb{T}} $. The graininess function $
 \mu~(\nu)  $ for a time scale $
 {\mathbb{T}} $ is defined by $
 \mu (t) = \sigma(t) - t (\nu (t) = t-\rho(t))$, and for every function $
 f:{\mathbb{T}} \to \mathbb{R} $ the notation $
 f^\sigma(f^\rho) $ means the composition $
 f \circ \sigma(f \circ \rho).$ We also need below the set $\mathbb{T}^\kappa$ which
 is derived from the time scale $\mathbb{T}$ as follows: If
 $\mathbb{T}$ has a left-scattered maximum $m$, then
 $\mathbb{T}^\kappa=\mathbb{T}-\{m\}$. Otherwise,
 $\mathbb{T}^\kappa=\mathbb{T}.$ Throughout this paper, we make the
blanket assumption that $a$ and $b$ are points in $\mathbb{T}$.
Often we assume $ a \leqslant b $. We then define the interval
$[a,b]$ in $\mathbb{T}$ by $[a,b]_\mathbb{T}=\{t\in{\mathbb{T}}: a
\leqslant t \leqslant b\}$.

For a function $
 f:\mathbb{T} \to \mathbb{R} $, the delta(nabla) derivative $
 f^\Delta(t)(f^\nabla(t)) $ at $
 t \in\mathbb{T} $ is defined to be the number (if it exists) such that for all $
 \varepsilon  > 0 $, there is a neighborhood $U$ of $t$ with
\begin{equation}
\begin{gathered}
 |f(\sigma \left( t \right)) - f\left( s \right) - f^\Delta  (t)(\sigma \left( t \right) - s)| < \varepsilon |\sigma \left( t \right) - s| \hfill \\
 (|f(\rho \left( t \right)) - f\left( s \right) - f^\nabla  (t)(\rho \left( t \right) - s)| < \varepsilon |\rho \left( t \right) - s|) \hfill \\
 \end{gathered}
\end{equation}
for all $
 s \in U $.
 If the delta(nabla) derivative $
 f^\Delta  (t) (f^\nabla(t))$ exits for all $
 t \in \mathbb{T} $, then we say that f is delta(nabla) differentiable on $
 \mathbb{T} $.
 We will make use of the following product and rules for the derivatives of the product $
 fg $ and the quotient $
 f/g $ (where $
 gg^\sigma(gg^\rho)   \ne 0 $) of two delta(nabla) differentiable functions $f$ and $g$,
\begin{equation}
\begin{gathered}
 (fg)^\Delta   = f^\Delta  g + f^\sigma  g^\Delta   = fg^\Delta   + f^\Delta  g^\sigma \hfill \\
 ((fg)^\nabla   = f^\nabla  g + f^\rho  g^\nabla   = fg^\nabla   + f^\nabla  g^\rho) \hfill \\
  \end{gathered}
\end{equation}

\begin{equation}
\begin{gathered}
 \left( {\frac{f}{g}} \right)^\Delta   = \frac{{f^\Delta  g - fg^\Delta  }}{{gg^\sigma  }} \hfill \\
\biggl( \left( {\frac{f}{g}} \right)^\nabla   = \frac{{f^\nabla  g - fg^\nabla }}{{gg^\rho  }}\biggr) \hfill \\
\end{gathered}
\end{equation}
Note that in the case $
 \mathbb{T} = \mathbb{R} $, we have $
 \sigma (t) = \rho (t) = t, \mu (t)= \nu (t)= 0$,
 $
 f^\Delta  (t)(f^\nabla(t)) = f'(t) $.
and in the case $
 \mathbb{T} = q\mathbb{Z} $, we have $
 \sigma (t) = t + q,\rho (t) = t - q,\mu (t)=\nu (t)=q$,
\begin{equation} f^\Delta  (t) =\frac{ f(t + q) - f(t)}{q}
\end{equation}
and
\begin{equation} f^\nabla  (t) =\frac{ f(t) - f(t-q)}{q}
\end{equation}
If $\mathbb{T}=q^{\mathbb{Z}}, q>1$, we have  $
 \sigma (t) = qt,\rho (t) = \frac{t}{q},\mu (t)=(q-1)t$,
\begin{equation} f^\Delta  (t) =\frac{ f(qt) - f(t)}{(q-1)t}, t\neq0
\end{equation}
and
\begin{equation} f^\nabla  (t) =\frac{ f(t) - f(t/q)}{\biggl(t-t/q\biggr)}, t\neq0
\end{equation}

A continuous function $f: \mathbb{T}\rightarrow\mathbb{R}$ is called
pre-differentiable with $D$, provided $D\subset{\mathbb{T}^\kappa}$,
$ \mathbb{T}^\kappa \backslash D $ is countable and contains no
right-scattered elements of $\mathbb{T}$, and $f$ is differentiable
at each $t\in{D}$. Let $f$ be rd(ld)-continuous. Then there exists a
function $F$ which is pre-differentiable with region of
differentiation $D$ such that $ F^\Delta  (x) = f(t)(F^\nabla  (x) = f(t))$ holds for all
$t\in{D}$. We define the Cauchy integral by
\begin{equation}
\begin{gathered}
\int\limits_b^c {f  (t)\Delta t = F(c) - F(b)}\hfill \\
\biggl(\int\limits_b^c {f  (t)\nabla t = F(c) - F(b)}\biggr) \hfill \\
\end{gathered}
\end{equation} where $F$ is a pre-antiderivative of $f$ and
$b,c\in\mathbb{T}$. The existence theorem \cite[p. 27, Theorem
1.74]{BP} reads as follows: Every rd(ld)-continuous function has an
antiderivative. In particular if $t_0\in\mathbb{T}$, then $F$
defined by $ F(t) = \int_{t_0 }^t {f(\tau )\Delta \tau }\biggl(F(t) = \int_{t_0 }^t {f(\tau )\nabla \tau } \biggr)$ is an
antiderivative of $f$.

If $f$ is delta(nabla) differentiable, then $f$ is continuous and
rd(ld)-continuous. We easily know that $$\sigma, \rho ,
f^\sigma (x), (f^\sigma (x))^p, f^\rho(x), (f^\rho (x))^p~~~~~p\in\mathbb{N} $$
are rd(ld)-continuous by using property of rd(ld)-continuous function. Thus,
all integrals involving main results of this paper are meaningful.
\subsection{Lemmas}

The following lemmas are useful and some of them can be found in the
book~\cite{BP}.

\begin{lemma}\cite[p.~423, Lemma~2.5]{YLQ}
Let $a, b\in{\mathbb{T}}$ and $p>1.$ Assume
$g:\mathbb{T}\to{\mathbb{R}}$ is delta differentiable at
$t\in{\mathbb{T}^\kappa}$ and non-negative, increasing function on
$[a, b]_{\mathbb{T}}$. Then \begin{equation} pg^{p - 1}(x) g^\Delta
(x) \leqslant (g^p (x))^\Delta   \leqslant p(g^\sigma  (x))^{p - 1}
g^\Delta  (x) .\end{equation}
\end{lemma}

\begin{lemma}\cite[p. 28, Theorem 1.76]{BP} If $f^{\Delta}(x)\geqslant0(f^{\nabla}(x)\geqslant0)$, then $f(x)$ is
nondecreasing.\end{lemma}

\begin{lemma}\cite[p.~5, Theorem~1.75]{BP}
Assume that $f:\mathbb{T}\to{\mathbb{R}}$ is rd-continuous at
$t\in{\mathbb{T}^\kappa}$. Then
\begin{equation}
\int_t^{\sigma (t)} {f(\tau )\Delta \tau  = } f(t)\mu (t).
\end{equation}
\end{lemma}

\begin{lemma}
Let $a, b\in{\mathbb{T}}$ and $p>1.$ Assume
$g:\mathbb{T}\to{\mathbb{R}}$ is nabla differentiable at
$t\in{\mathbb{T}^\kappa}$ and non-negative, increasing function on
$[a, b]_{\mathbb{T}}$. Then \begin{equation} p(g^\rho  (x))^{p - 1} g^\nabla
(x) \leqslant (g^p (x))^\nabla   \leqslant p g^{p - 1}(x)
g^\nabla  (x) .\end{equation}
\end{lemma}
\begin{proof}
Using (2.2), we have
\begin{equation*}
\left( {g^2 } \right)^\nabla   = \left( {g + g^\rho  } \right)g^\nabla.
\end{equation*}
So, we easily obtain
\begin{equation*}
\left( {g^p } \right)^\nabla   = \left( {g^{p - 1}  + g^\rho  g^{p - 2}  +  \cdots  + \left( {g^\rho  } \right)^{p - 1} } \right)g^\nabla  .
\end{equation*}
 by mathematical induction. Considering property of the function $g$, the proof is completed.
\end{proof}

 For more discussion on time scales, we refer the reader to \cite{BP}.
\section{ Main Results}

\begin{theorem}
Let $a, b\in{\mathbb{T}}$ and $t\geqslant3.$ Assume
$f,\sigma:\mathbb{T}\to{\mathbb{R}}$ be delta differentiable at
$t\in{\mathbb{T}^\kappa}.$ If $f$ is a non-negative, increasing
function on $[a, b]_{\mathbb{T}}$ and satisfies
\begin{equation}
f^{t - 2} (x)f^\Delta  (x) \geqslant (t - 2)(f^{\sigma ^2 } (x))^{t
- 2} (\sigma ^2 (x) - a)^{t - 3} \sigma ^\Delta  (x)
\end{equation}
where $ \sigma ^2 (x) = \sigma (\sigma (x)) $. Then \begin{equation}\begin{gathered}
\int_a^b {f^t (x)\Delta x - \left(\int_a^b {f(x)\Delta x}
\right)^{t - 1} }\hfill \\
\geqslant
f^{t - 2} (a)\left[ {f(a) - (t - 1)\mu ^{t - 2} (a)} \right]\int_a^b {f(x)\Delta x}.\end{gathered}
\end{equation}
\end{theorem}
\begin{proof}
Define $$ F(x) = \int_a^x {f^t (u)\Delta u - \left(\int_a^x
{f(u)\Delta u} \right)^{t - 1} }
$$ and $
g(x) = \int_a^x {f(u)\Delta u}. $ It is easy to see
$g^{\Delta}{(x)}=f(x)$. Using Lemma 2.1, it follows that $$
\begin{gathered}
  F^\Delta  (x)  \geqslant f^t (x) - (t- 1)(g^\sigma  (x))^{t - 2} g^\Delta  (x)   \hfill \\
   = f(x)F_1 (x) \hfill \\
\end{gathered}
$$ where $
F_1 (x) = f^{t - 1} (x) - (t - 1)(g^\sigma  (x))^{t - 2} $.

Using Lemma 2.1 again, we have $$ F_1^\Delta  (x) \geqslant (t
- 1)f^{t - 2} (x)f^\Delta  (x) - (t - 1)(t - 2)(g^{\sigma ^2 }
(x))^{t - 3} f^\sigma  (x)\sigma ^\Delta (x).
$$
Since $f$ is a non-negative and increasing function, then
\begin{equation} g^{\sigma ^2 } (x) = \int_a^{\sigma ^2 (x)}
{f(u)\Delta u} \leqslant f^{\sigma ^2 } (x)(\sigma ^2 (x) - a).
\end{equation} Hence, $$
\begin{gathered}
 F_1 ^\Delta  (x) \hfill \\
 \geqslant (t - 1)[f^{t - 2} (x)f^\Delta  (x) - (t - 2)(f^{\sigma ^2 } (x))^{t - 3} f^\sigma  (x)(\sigma ^2 (x) - a)^{t - 3} \sigma ^\Delta  (x)] \hfill \\
   \geqslant (t - 1)[f^{t - 2} (x)f^\Delta  (x) - (t - 2)(f^{\sigma ^2 } (x))^{t - 2} (\sigma ^2 (x) - a)^{t - 3} \sigma ^\Delta  (x)]  \hfill \\
   \geqslant 0 \hfill \\
\end{gathered}
$$ By Lemma 2.2, we conclude that $F_1 (x)$ is an increasing function.
Hence, $$ F_1 (x) \geqslant F_1 (a)=f^{t-2}(a)[f(a)-(t-1)\mu^{t-2}(a)]
$$ which means that $$F^{\Delta}(x)\geqslant f^{t-2}(a)[f(a)-(t-1)\mu^{t-2}(a)]f(x)$$ by applying Lemma 2.3. It follows that $$ \biggl(F(x)-f^{t-2}(a)[f(a)-(t-1)\mu^{t-2}(a)]g(x)\biggr)^\Delta
\geqslant 0. $$ Thus, we have $$\begin{gathered}F(b)-f^{t-2}(a)[f(a)-(t-1)\mu^{t-2}(a)]g(b)\hfill \\
\geqslant F(a)-f^{t-2}(a)[f(a)-(t-1)\mu^{t-2}(a)]g(a)\hfill \\
=0.\hfill \\\end{gathered}$$
This finish the proof.
\end{proof}

\begin{remark}
If $\mathbb{T}=\mathbb{R}$ and $f(a)\neq0$ in Theorem 3.1, we deduce
Theorem 2.1 in \cite{K}.
\end{remark}

\begin{remark}
If $\mathbb{T}=h\mathbb{Z}$  in Theorem 3.1, Theorem 3.1 generalizes
Theorem 3.2 in \cite{R}.
\end{remark}

\begin{theorem}
Let $a, b\in{\mathbb{T}}$ and $p\geqslant1.$ Assume
$f,\sigma:\mathbb{T}\to{\mathbb{R}}$ be delta differentiable at
$t\in{\mathbb{T}^\kappa}.$ If $f$ is a non-negative, increasing
function on $[a, b]_{\mathbb{T}}$ and satisfies \begin{equation}
f^{p} (x)f^\Delta (x) \geqslant \frac{p}{(b-a)^{p-1}}\biggl(f^{\sigma^2 }
(x)\biggr)^{p} \biggl(\sigma ^2 (x) - a\biggr)^{p - 1} \sigma ^\Delta
(x).\end{equation} Then
\begin{equation} \begin{gathered}\int_a^b {f^{p+2} (x)\Delta x-
\frac{1}{(b-a)^{p-1}}\left(\int_a^b {f(x)\Delta x} \right)^{p+1}
}\hfill \\
\geqslant f^p (a)\left[ {f(a) - \frac{{p + 1}}{{(b - a)^{p - 1} }}\mu ^p (a)} \right]\int_a^b {f(x)\Delta x}
.\end{gathered}\end{equation}
\end{theorem}

\begin{proof}
Define $$ G(x) = \int_a^x {f^{p+2} (t)\Delta t -
\frac{1}{(b-a)^{p-1}}\left(\int_a^x {f(t)\Delta t} \right)^{p+1} }
$$ and $
g(x) = \int_a^x {f(t)\Delta t}. $ Using Lemma 2.1, it follows that
$$
\begin{gathered}
  G^\Delta  (x) = f^{p+2} (x) - \frac{1}{(b-a)^{p-1}}(g^{p + 1} (x))^\Delta   \hfill \\
   \geqslant f^{p+2} (x) - \frac{p+1}{(b-a)^{p-1}}(g^\sigma  (x))^{p} g^\Delta  (x)  \hfill \\
   \geqslant f(x)\left[f^{p+1} (x) - \frac{p+1}{(b-a)^{p-1}}(g^\sigma  (x))^{p}\right]\hfill \\
   =f(x)G_1 (x)  \hfill \\
\end{gathered}
$$ where $
G_1 (x) = f^{p +1} (x) - \frac{p+1}{(b-a)^{p-1}}(g^\sigma  (x))^{p} $.

Using Lemma 2.1 and (3.3), we have $$ \begin{gathered}
G_1^\Delta (x)\hfill \\ 
\geqslant (p + 1)f^{p } (x)f^\Delta  (x) -
\frac{p(p+1)}{(b-a)^{p-1}}\biggl(g^{\sigma ^2
} (x)\biggr)^{p - 1} f^{\sigma}(x)\sigma ^\Delta (x)\hfill \\
\geqslant (p + 1)\left[f^{p } (x)f^\Delta  (x) -
\frac{p}{(b-a)^{p-1}}\biggl(f^{\sigma ^2 } (x)\biggr)^{p} \biggl(\sigma ^2 (x) - a\biggr)^{p
- 1}\sigma ^\Delta (x)\right]\hfill \\
\geqslant0\hfill \\\end{gathered}.
$$ Similar to the proof of Theorem 3.1, we have $$
\begin{array}{l}
 G^\Delta  (x) \ge f(x)G_1 (a) \\
  \Leftrightarrow \left( {G(x) - g(x)G_1 (a)} \right)^\Delta   \geqslant 0 \\
 \end{array}
$$ which implies $$
G(x) - g(x)G_1 (a) \geqslant G(a) - g(a)G_1 (a) = 0.
$$The
proof is complete.
\end{proof}

\begin{remark}
If $\mathbb{T}=\mathbb{R}$ and $f(a)\neq0$ in Theorem 3.2, we deduce Theorem 2.2 in
\cite{K}.
\end{remark}

\begin{remark}
If $\mathbb{T}=h\mathbb{Z}$  in Theorem 3.2, Theorem 3.2 generalizes
Theorem 3.3 in \cite{R}.
\end{remark}

\begin{theorem}
Let $a, b\in{\mathbb{T}}$ and $p\geqslant3.$ Assume
$f,\sigma:\mathbb{T}\to{\mathbb{R}}$ be delta differentiable at
$t\in{\mathbb{T}^\kappa}.$ If $f$ is a non-negative, increasing
function on $[a, b]_{\mathbb{T}}$ and satisfies \begin{equation}
f^{p - 3} (x)f^\Delta  (x) \geqslant (p - 2)\biggl(f^{\sigma^2 } (x)\biggr)^{p -
3} \biggl(\sigma ^2 (x) - a\biggr)^{p - 3} \sigma ^\Delta  (x).\end{equation}
 Then
\begin{equation}
\begin{gathered} \int_a^b {f^p (x)\Delta x}-
\left(\int_a^b {f^{\rho}(x)\Delta x} \right)^{p - 1}\hfill \\
\geqslant (f (a))^{p-2}[f(a) - (p - 1)\mu ^{p - 2}(a) ]\int_a^b {f(\rho(x))\Delta x}.\hfill \\
\end{gathered}
\end{equation}
\end{theorem}

\begin{proof}
Define $$ H(x) = \int_a^x {f^p (t)\Delta t - \left(\int_a^x
{f^{\rho}(t)\Delta t} \right)^{p - 1} }
$$ and $
g(x) = \int_a^x {f^{\rho}(t)\Delta t}. $  Using Lemma 2.1, it follows that $$
\begin{gathered}
  H^\Delta  (x) = f^p (x) - (g^{p - 1} (x))^\Delta   \hfill \\
   \geqslant f^p (x) - (p - 1)(g^\sigma  (x))^{p - 2} g^\Delta  (x) \hfill \\
   \geqslant f(\rho(x))H_1 (x) \hfill \\
\end{gathered}
$$ where $
H_1 (x) = f^{p - 1} (x) - (p - 1)(g^\sigma  (x))^{p - 2} $.

Using Lemma 2.1 and (3.3) again, we have $$
\begin{gathered}
  H_1 ^\Delta  (x) \geqslant (p - 1)f^{p - 2} (x)f^\Delta  (x) - (p - 1)(p - 2)\biggl(g^{\sigma ^2 } (x)\biggr)^{p - 3} f(x)(\sigma (x))^\Delta   \hfill \\
   \geqslant (p - 1)f(x)\biggl[f^{p - 3} (x)f^\Delta  (x) - (p - 2)\biggl(f^{\sigma^2}  (x)\biggr)^{p - 3} \biggl(\sigma ^2 (x) - a\biggr)^{p - 3} \sigma ^\Delta  (x)\biggr] \geqslant 0 \hfill \\
\end{gathered}$$ By Lemma 2.2, we conclude that $H_1 (x)$ is an increasing function.
Hence, $$
\begin{gathered}
  H_1 (x) \geqslant
   H_1 (a) = f^{p - 1} (a) - (p - 1)(g^\sigma  (a))^{p - 2}  \hfill \\
   = (f (a))^{p-2}[f(a) - (p - 1)\mu ^{p - 2}(a) ] \hfill \\
\end{gathered}
$$ which means that $(H(x) - g(x)H_1 (a))^\Delta \geqslant0.$ The proof is complete.
\end{proof}

\begin{theorem}
Let $a, b\in{\mathbb{T}}$ and $p\geqslant1.$ Assume
$f,\sigma:\mathbb{T}\to{\mathbb{R}}$ be delta differentiable at
$t\in{\mathbb{T}^\kappa}.$ If $f$ is a non-negative, increasing
function on $[a, b]_{\mathbb{T}}$ and satisfies \begin{equation}
(f^\sigma (x))^\Delta   \geqslant p\sigma ^\Delta  (x),
\end{equation} then

\begin{equation}
\begin{gathered}
\int_a^b
{(f^\sigma (x))^{p + 2} \Delta x}- \frac{1} {{(b - a)^{p -
1} }}\left(\int_a^b {f^\rho (x)\Delta x} \right)^{p +
1}\hfill \\
\geqslant \biggl[(f^\sigma  (a))^{p + 1}  - \frac{{p + 1}}
{{(b - a)^{p - 1} }}(f^\rho  (a)\mu (a))^p\biggr] \int_a^b {f(\rho(x))\Delta x}.  \hfill \\
\end{gathered}
\end{equation}
\end{theorem}

\begin{proof}
Define $$ W(x)=\int_a^x {(f^{\sigma} (t))^p\Delta t - \left(\int_a^x
{f^{\rho}(t)\Delta t} \right)^{p - 1} }$$ and $ g(x) = \int_a^x
{f^{\rho}(t)\Delta t}. $ Using Lemma 2.1, it follows that $$
\begin{gathered}
  W^\Delta  (x) \geqslant (f^\sigma  (x))^{p + 2}  - \frac{{p + 1}}
{{(b - a)^{p - 1} }}(g^\sigma  (x))^p g^\Delta  (x) \hfill \\
   \geqslant f^\sigma  (x)\left[f^\sigma  (x))^{p + 1}  - \frac{{p + 1}}
{{(b - a)^{p - 1} }}(g^\sigma  (x))^p \right] \hfill \\
\geqslant f(\rho(x))W_1 (x) \hfill \\
\end{gathered}$$ where $
W_1 (x) = (f^\sigma  (x))^{p + 1}  - \frac{{p + 1}} {{(b - a)^{p - 1}
}}(g^\sigma  (x))^p$.

Using Lemma 2.1 again, we have $$\begin{gathered}
  W_1 ^\Delta  (x) \hfill \\
  \geqslant (p + 1)\left[\biggl(f^\sigma  (x)\biggr)^p \biggl(f^\sigma  (x)\biggr)^\Delta   - \frac{{p}}
{{(b - a)^{p - 1} }}\biggl(g^{\sigma ^2 } (x)\biggr)^{p - 1} f(x)(\sigma
(x))^\Delta  \right]\hfill \\
\end{gathered}
$$Since $f$ is a non-negative and increasing function, then \begin{equation}
g^{\sigma ^2 } (x) = \int_a^{\sigma ^2 (x)} {f^{\rho}(t)\Delta t}
\leqslant f^{\rho\sigma ^2 } (x)(\sigma ^2 (x) - a)\leqslant
f^\sigma (x)(b-a).\end{equation} Hence, $$ W_1 ^{\Delta}(x)\geqslant (p
+ 1)\biggl(f^\sigma  (x)\biggr)^p \biggl[\biggl(f^\sigma  (x)\biggr)^\Delta   - p(\sigma
(x))^\Delta\biggr ].$$By Lemma 2.2, we conclude that $W_1 (x)$ is an
increasing function. Hence, $$
\begin{gathered}
  W_1 (x) \geqslant W_1 (a)\hfill \\
  = (f^\sigma  (a))^{p + 1}  - \frac{{p + 1}}
{{(b - a)^{p - 1} }}(g^\sigma  (a))^p  \hfill \\
   = (f^\sigma  (a))^{p + 1}  - \frac{{p + 1}}
{{(b - a)^{p - 1} }}(f^\rho  (a)\mu (a))^p  \hfill \\
\end{gathered}$$ which means that $(W(x) - g(x)W_1 (a))^\Delta \geqslant0.$ The proof is complete.
\end{proof}

\begin{theorem}
Let $a, b\in{\mathbb{T}}$ and $p\geqslant3.$ Assume
$f,\sigma:\mathbb{T}\to{\mathbb{R}}$ be delta differentiable at
$t\in{\mathbb{T}^\kappa}.$ If $f$ is a non-negative, increasing
function on $[a, b]_{\mathbb{T}}$ and satisfies
\begin{equation}
 \biggl(f^\sigma  (x)\biggr)^\Delta  \geqslant (p - 2)\biggl(\sigma ^2 (x) - a\biggr)^{p - 3} \sigma ^\Delta  (x)
\end{equation}
Then
\begin{equation}
\begin{gathered} \int_a^b {f^{\sigma} (x))^{p}\Delta x}-
\left(\int_a^b {f^{\rho}(x)\Delta x} \right)^{p - 1}\hfill \\
\geqslant (f^\sigma  (a))^{p - 2} \biggl[f^\sigma  (a) - (p - 1)\mu^{p - 2}(a)\biggr] \int_a^b {f(\rho(x))\Delta x}.\hfill \\
\end{gathered}
\end{equation}
\end{theorem}

\begin{proof}
Define $$ Q(x)=\int_a^x {(f^{\sigma} (t))^p\Delta t - \left(\int_a^x
{f^{\rho}(t)\Delta t} \right)^{p - 1} }$$ and $ g(x) = \int_a^x
{f^{\rho}(t)\Delta t}. $ Using Lemma 2.1, it follows that $$
\begin{gathered}
  Q^\Delta  (x) = (f^\sigma (x))^p - (g^{p - 1} (x))^\Delta   \hfill \\
   \geqslant (f^\sigma (x))^p - (p - 1)(g^\sigma  (x))^{p - 2} g^\Delta  (x) \hfill \\
   \geqslant f(\rho(x))Q_1 (x) \hfill \\
\end{gathered}
$$ where $
Q_1 (x) = (f^\sigma (x))^{p-1} - (p - 1)(g^\sigma  (x))^{p - 2} $.

Using Lemma 2.1 and (3.10) again, we have $$
\begin{gathered}
  Q_1 ^\Delta  (x)\hfill \\ 
  \geqslant (p - 1)[(f^\sigma  (x))^{p - 2} (f^\sigma  (x))^\Delta   - (p - 2)(g^{\sigma ^2 } (x))^{p - 3} (g^\sigma  (x))^\Delta  ] \hfill \\
   \geqslant (p - 1)\biggl(f^\sigma  (x)\biggr)^{p - 2}\biggl[ (f^\sigma  (x))^\Delta   - (p - 2)\biggl(\sigma ^2 (x) - a\biggr)^{p - 3} \sigma ^\Delta  (x)\biggr] \hfill \\
   \geqslant 0 \hfill \\
\end{gathered}.$$ By Lemma 2.2, we conclude that $Q_1(x)$ is an increasing function.
Hence, $$
\begin{gathered}
  Q_1 (x) \geqslant Q_1 (a) = (f^\sigma  (a))^{p - 1}  - (p - 1)(g^\sigma  (a))^{p - 2}  \hfill \\
   \geqslant (f^\sigma  (a))^{p - 2} \biggl(f^\sigma  (a) - (p - 1)\mu^{p - 2}(a)\biggr)  \hfill \\
\end{gathered}
$$ which means that  $$\biggl(Q(x) - g(x)(f^\sigma  (a))^{p - 2} \biggl(f^\sigma  (a) - (p - 1)\mu^{p - 2}(a)\biggr)\biggr)^\Delta \geqslant0.$$ The proof is complete.
\end{proof}

Next, we generalized Feng Qi type inequalities related to nabla derivative.
\begin{theorem}
Let $a, b\in{\mathbb{T}}$ and $t\geqslant3.$ Assume
$f:\mathbb{T}\to{\mathbb{R}}$ be nabla differentiable at
$t\in{\mathbb{T}^\kappa}.$ If $f$ is a non-negative, increasing
function on $[a, b]_{\mathbb{T}}$ and satisfies
\begin{equation}
\biggl(f^\rho (x)\biggr)^{t - 2}f^\nabla  (x) \geqslant (t - 2) (f(x))^{t - 2}(x - a)^{t - 3}. 
\end{equation}
 Then \begin{equation}
\int_a^b {f^t (x)\nabla x - \left(\int_a^b {f(x)\nabla x}
\right)^{t - 1} }
\geqslant
f^{t - 1} (a)\int_a^b {f(x)\nabla x}.
\end{equation}
\end{theorem}
\begin{proof}
Define $$ F(x) = \int_a^x {f^t (u)\nabla u - \left(\int_a^x
{f(u)\nabla u} \right)^{t - 1} }
$$ and $
g(x) = \int_a^x {f(t)\nabla t}. $ It is easy to see
$g^{\nabla}{(x)}=f(x)$. Using Lemma 2.4, it follows that $$
\begin{gathered}
  F^\nabla  (x)\hfill \\ 
   \geqslant f^t (x) - (t- 1)(g (x))^{t - 2} g^\nabla  (x)   \hfill \\
   = f(x)F_1 (x) \hfill \\
\end{gathered}
$$ where $
F_1 (x) = f^{t - 1} (x) - (t - 1)(g(x))^{t - 2} $.

Using Lemma 2.4 again, we have $$ F_1^\nabla  (x) \geqslant (t
- 1)(f^\rho (x))^{t - 2} (x)f^\nabla  (x) - (t - 1)(t - 2)(g
(x))^{t - 3} f(x).
$$
Since $f$ is a non-negative and increasing function, then
\begin{equation} g (x) = \int_a^{ x}
{f(t)\nabla t} \leqslant f (x)(x - a).
\end{equation} Hence, $$
\begin{gathered}
 F_1 ^\nabla  (x) \hfill \\
 \geqslant (t - 1)\biggl[(f^{\rho } (x))^{t - 2}f^\nabla  (x) - (t - 2)(f (x))^{t - 2} (x - a)^{t - 3}\biggr ] \hfill \\
    \geqslant 0 \hfill \\
\end{gathered}
$$ By Lemma 2.2, we conclude that $F_1 (x)$ is an increasing function.
Hence, $$ F_1 (x) \geqslant F_1 (a)
$$ which means that $$F^{\nabla}(x)\geqslant  F_1 (a)f(x).$$ It follows that $$ \biggl(F(x)- F_1 (a)g(x)\biggr)^\nabla
\geqslant 0. $$ Thus, we have $$F(b)-f^{t - 1} (a)g(b)
\geqslant F(a)-f^{t -1} (a)g(a)
=0.$$
This finish the proof.
\end{proof}

\begin{theorem}
Let $a, b\in{\mathbb{T}}$ and $p\geqslant1.$ Assume
$f:\mathbb{T}\to{\mathbb{R}}$ be nabla differentiable at
$t\in{\mathbb{T}^\kappa}.$ If $f$ is a non-negative, increasing
function on $[a, b]_{\mathbb{T}}$ and satisfies \begin{equation}
(f^{\rho} (x))^{p}f^\nabla (x) \geqslant \frac{p}{(b-a)^{p-1}}(f^{p }
(x)) (x - a)^{p - 1}.\end{equation} then
\begin{equation}\int_a^b {f^{p+2} (x)\nabla x-
\frac{1}{(b-a)^{p-1}}\left(\int_a^b {f(x)\nabla x} \right)^{p+1}
}
\geqslant f^{p+1} (a)\int_a^b {f(x)\nabla x}
.\end{equation}
\end{theorem}

\begin{proof}
Define $$ G(x) = \int_a^x {f^{p+2} (t)\nabla t -
\frac{1}{(b-a)^{p-1}}\left(\int_a^x {f(t)\nabla t} \right)^{p+1} }
$$ and $
g(x) = \int_a^x {f(t)\nabla t}. $ Using Lemma 2.4, it follows that
$$
\begin{gathered}
  G^\nabla  (x) = f^{p+2} (x) - \frac{1}{(b-a)^{p-1}}(g^{p + 1} (x))^\nabla   \hfill \\
   \geqslant f^{p+2} (x) - \frac{p+1}{(b-a)^{p-1}}g^{p}  (x) g^\nabla  (x)  \hfill \\
   \geqslant f(x)\left[f^{p+1} (x) - \frac{p+1}{(b-a)^{p-1}}g^{p} (x)\right]\hfill \\
   =f(x)G_1 (x)  \hfill \\
\end{gathered}
$$ where $
G_1 (x) = f^{p +1} (x) - \frac{p+1}{(b-a)^{p-1}}g^{p} (x) $.

Using Lemma 2.4 again, we have $$ \begin{gathered}
G_1^\nabla (x)\hfill \\
\geqslant (p + 1)\left[(f^{\rho } (x))^{p}f^\nabla  (x) -
\frac{p}{(b-a)^{p-1}}f^{p } (x) (x - a)^{p
- 1}\right]\hfill \\
\geqslant0\hfill \\\end{gathered}.
$$ Similar to the proof of Theorem 3.6, we have $$
\begin{array}{l}
 G^\nabla  (x) \ge f(x)G_1 (a) \\
  \Leftrightarrow \left( {G(x) - g(x)G_1 (a)} \right)^\nabla   \geqslant 0 \\
 \end{array}
$$ which implies $$
G(x) - g(x)G_1 (a) \geqslant G(a) - g(a)G_1 (a) = 0.
$$The
proof is completed.
\end{proof}
\begin{remark}
Similar Theorem 3.4, Theorem 3.5 and Theorem 3.6, we easily obtain similar Feng Qi type inequalities related to the nabla derivative. We omit the details for the sake of simplicity.
\end{remark}

\end{document}